\documentclass[reqno]{amsart}
\usepackage{amsmath, amssymb, amsthm, epsfig}
\usepackage{hyperref, latexsym}
\usepackage{url} 
\def\today{\ifcase\month\or
  January\or February\or March\or April\or May\or June\or
  July\or August\or September\or October\or November\or December\fi
  \space\number\day, \number\year}

\DeclareMathOperator{\sgn}{\mathrm{sgn}}

 \newtheorem{theorem}{Theorem}[section]
 \newtheorem{lemma}[theorem]{Lemma}
 
 \newtheorem{corollary}[theorem]{Corollary}
 \theoremstyle{definition}

 \theoremstyle{remark}
 
 \newcommand{\mc}{\mathcal}
 \newcommand{\A}{\mc{A}}

 \newcommand{\J}{\mc{J}}
 \newcommand{\K}{\mc{K}}
 \newcommand{\KK}{K^{\star}}

 \newcommand{\rr}{\mc{R}}

 \newcommand{\C}{\mathbb{C}}
 \newcommand{\R}{\mathbb{R}}

 \newcommand{\Z}{\mathbb{Z}}
 \newcommand{\csch}{\mathrm{csch}}
 \newcommand{\sech}{\mathrm{sech}}
 \newcommand{\tF}{\widehat{F}}
 
 \newcommand{\ttK}{\widetilde{K}}

 \newcommand{\uk}{\widehat{k}}

  \newcommand{\wk}{\widetilde{k}}
 
\newcommand{\wv}{\widetilde{v}}
 
 \newcommand{\tK}{\widehat{K}}

 \newcommand{\tq}{\widehat{q}}

 \newcommand{\h}{\frac12}
 \newcommand{\hh}{\tfrac12}
 
 \newcommand{\dt}{\text{\rm d}t}
 \newcommand{\du}{\text{\rm d}u}
 \newcommand{\dv}{\text{\rm d}v}
 \newcommand{\dw}{\text{\rm d}w}
 \newcommand{\dx}{\text{\rm d}x}
 \newcommand{\dy}{\text{\rm d}y}
 \newcommand{\dl}{\text{\rm d}\lambda}
 \newcommand{\dmu}{\text{\rm d}\mu(\lambda)}
 \newcommand{\dnu}{\text{\rm d}\nu(\lambda)}
 
 \newcommand{\dnnu}{\text{\rm d}\nu_n(\lambda)}

\begin{document}

\title[Extremal functions]{Some extremal functions in Fourier analysis, III}

\author[Carneiro and Vaaler]{Emanuel Carneiro* and Jeffrey D. Vaaler**}
\address{Department of Mathematics, University of Texas at Austin, Austin, TX 78712-1082.}
\email{ecarneiro@math.utexas.edu}
\address{Department of Mathematics, University of Texas at Austin, Austin, TX 78712-1082.}
\email{vaaler@math.utexas.edu}
\allowdisplaybreaks
\thanks{*Research supported by CAPES/FULBRIGHT grant BEX 1710-04-4.}
\thanks{**Research supported by the National Science Foundation, DMS-06-03282.}

\numberwithin{equation}{section}


\subjclass[2000]{Primary 41A30, 41A52, 42A05. Secondary 41A05, 41A44, 42A10.}

\date{\today}

\keywords{approximation, entire functions, exponential type}

\begin{abstract}
We obtain the best approximation in $L^1(\R)$, by entire functions of exponential type, for a class of even functions that includes $e^{-\lambda|x|}$, where $\lambda >0$, $\log |x|$ and $|x|^{\alpha}$, where $-1 < \alpha < 1$. We also give periodic versions of these results where the approximating functions are trigonometric polynomials of bounded degree.
\end{abstract}

\maketitle

\section{Introduction}

An entire function $K:\C \to \C$ is of exponential type $\sigma \geq 0$ if, for any $\epsilon >0$, there exists a constant $C_{\epsilon}$ such that for all $z \in \C$ we have
$$|K(z)| \leq C_{\epsilon} e^{(\sigma + \epsilon)|z|}\,.$$
Given a function $f:\R\to\R$ we address here the problem of finding an entire function $K(z)$ of exponential type at most $\pi$ such that the integral
\begin{equation} \label{prob}
 \int_{-\infty}^{\infty} |K(x) - f(x)| \, \dx
\end{equation}
is minimized. A typical variant of this problem occurs when we impose the additional condition that $K(z)$ is real on $\R$ and satisfies $K(x) \geq f(x)$ for all $x \in \R$. In this case a minimizer of the integral (\ref{prob}) is called an extremal majorant of $f(x)$. Extremal minorants are defined analogously.\\

The study of these extremal functions dates back to A. Beurling in the 1930's, who solved the problem (\ref{prob}) (and its majorizing version) for $f(x) = \sgn(x)$. A complete collection of his results and many applications to analytic number theory (including Selberg's proof of the large sieve inequality) can be found in the paper \cite{V} by J.D. Vaaler. In \cite{GV}, Graham and Vaaler constructed the extremal majorants and minorants for the function $f(x) = e^{-\lambda |x|}$, where $\lambda >0$. Recently, Carneiro and Vaaler in \cite{CV} were able to extend the construction of extremal majorants for a wide class of even functions that includes $\log |x|$ and $|x|^{\alpha}$, where $-1 < \alpha < 1$. The case $f(x) = \log |x|$, which can be viewed as a Fourier conjugate of $f(x) = \sgn(x)$, is particularly important, providing a number of interesting applications. Other problems on approximation and majorization by entire functions have been discussed in \cite{Lit}, \cite{Lit2}, \cite{M} and \cite{S}. Extensions of the problem to several variables are considered in \cite{BMV}, \cite{HV} and \cite{LV}.\\

The purpose of this paper, the third in this series, is to settle the best approximation problem (\ref{prob}) for the function $f(x) = e^{-\lambda|x|}$, where $\lambda>0$, and also for the same class of even functions considered in \cite{CV}, that includes $\log |x|$ and $|x|^{\alpha}$, where $-1 < \alpha < 1$.

We start by defining the entire function $z \mapsto K(\lambda,z)$ by

\begin{equation}\label{Intro1}
K(\lambda, z) = \left(\frac{\cos \pi z}{\pi}\right)\left\{ \sum_{n \in \Z} \frac{ (-1)^{n}\, e^{-\lambda\left|n- \tfrac{1}{2}\right|}}{\left(z - n + \tfrac{1}{2}\right)} \right\}.
\end{equation}
This function has exponential type $\pi$ and interpolates $e^{-\lambda|x|}$ at the integers plus a half. The construction of such functions and how they appear as natural candidates to our problem are explained in \cite[sections 2 and 3]{V}. Our first result is the following.
\begin{theorem}\label{thm1.1}
The function $K(\lambda, z)$ defined in {\rm(\ref{Intro1})} satisfies the following extremal property
\begin{itemize}
 \item[(i)] If $\ttK(z)$ is an entire function of exponential type at most $\pi\delta$, where $\delta>0$, then
\begin{equation}\label{Intro2}
\int_{-\infty}^{\infty}\left|e^{-\lambda|x|} - \ttK(x)\right|  \dx \geq \tfrac{2}{\lambda} - \tfrac{2}{\lambda}\,\sech\left(\tfrac{\lambda}{2\delta}\right),
\end{equation}
with equality if and only if $\ttK(z) = K(\delta^{-1}\lambda,\delta z)$. 

\item[(ii)] For $x \in \R$ we have
\begin{equation}\label{Intro3}
 \sgn(\cos \pi x) = \sgn\left\{ e^{-\lambda |x|} - K(\lambda,x ) \right\}.
\end{equation}
\end{itemize}
\end{theorem}
From Theorem \ref{thm1.1} we see that $x \mapsto K(\lambda,x)$ is integrable on $\R$. Its Fourier transform
\begin{equation}\label{Intro3.1}
\tK(\lambda,t) = \int_{-\infty}^{\infty}K(\lambda,x)e(-tx)\, \dx
\end{equation}
is a continuous function of the real variable $t$ supported on the interval $[-\hh, \hh]$. Here we write $e(z) = e^{2\pi i z}$. The Fourier transform in (\ref{Intro3.1}) is a nonnegative function of $t$ and is given explicitly in Lemma \ref{lem3.3}.

The description of the sign changes given by (\ref{Intro3}) is a key point in our argument. It will allow us to apply the techniques of \cite{CV} when we integrate with respect to the parameter $\lambda$. For this, let $\mu$ be a measure defined on the Borel subsets of $(0,\infty)$ such that
\begin{equation}\label{Intro4}
\int_0^{\infty}\frac{\lambda}{\lambda^2 + 1} \, \dmu< \infty.
\end{equation}
It follows from (\ref{Intro4}) that, for $x\not= 0$, the function
\begin{equation*}\label{Intro5}
\lambda\mapsto e^{-\lambda |x|} - e^{-\lambda}
\end{equation*}
is integrable on $(0,\infty)$ with respect to $\mu$. We define then $f_{\mu}:\R\rightarrow \R\cup\{\infty\}$ by
\begin{equation}\label{Intro6}
f_{\mu}(x) = \int_0^{\infty} \big\{e^{-\lambda |x|} - e^{-\lambda}\big\}\, \dmu,
\end{equation}
where 
\begin{equation*}\label{Intro7}
f_{\mu}(0) = \int_0^{\infty} \big\{1 - e^{-\lambda}\}\, \dmu
\end{equation*}
may take the value $\infty$. Using $f_{\mu}$, we define $K_{\mu}:\C\rightarrow\C$ by
\begin{equation}\label{Intro8}
K_{\mu}(z) = \lim_{N\rightarrow\infty} \left(\frac{\cos \pi z}{\pi}\right) 
	\left\{\sum_{n=-N}^{N+1} \frac{(-1)^{n}f_{\mu}(n-\hh)}{(z - n + \hh)} 
		\right\}.
\end{equation}
We will show that the sequence on the right of (\ref{Intro8}) converges uniformly on compact subsets of $\C$ and therefore
defines $K_{\mu}(z)$ as a real entire function.  Then it is easy to check that $K_{\mu}$ interpolates the values of 
$f_{\mu}$ at real numbers $x$ such that $x - \h$ is an integer.  That is, the identity
\begin{equation}\label{Intro9}
K_{\mu}(n - \hh) = f_{\mu}(n - \hh)
\end{equation}
holds for each integer $n$. We will prove that the entire function $K_{\mu}(z)$ satisfies the following extremal property.

\begin{theorem}\label{thm1.2}
Assume that the measure $\mu$ satisfies {\rm (\ref{Intro4})}.  
\begin{itemize}
\item[(i)]  The real entire function $K_{\mu}(z)$ defined by {\rm (\ref{Intro8})} has exponential type at most $\pi$.
\item[(ii)]  For real $x\not= 0$ the function
\begin{equation*}
\lambda \mapsto e^{-\lambda |x|} - K(\lambda, x)
\end{equation*}
is integrable on $(0, \infty)$ with respect to $\mu$. 
\item[(iii)]  For all real $x$ we have
\begin{equation}\label{Intro10}
f_{\mu}(x) - K_{\mu}(x) = \int_0^{\infty} \big\{e^{-\lambda |x|} - K(\lambda, x)\big\}\, \dmu.
\end{equation}
\item[(iv)]  The function $x\mapsto f_{\mu}(x) - K_{\mu}(x)$ is integrable on $\R$, and
\begin{equation}\label{Intro11}
\int_{-\infty}^{\infty} \left|f_{\mu}(x) - K_{\mu}(x)\right|\,\dx
		= \int_0^{\infty}\big\{\tfrac{2}{\lambda} - \tfrac{2}{\lambda}\sech\bigl(\tfrac{\lambda}{2}\bigr)\big\}\, \dmu.
\end{equation}
\item[(v)]  If $t\not= 0$ then
\begin{align}\label{Intro12}
\begin{split}
\int_{-\infty}^{\infty} \big\{f_{\mu}(x) - &K_{\mu}(x)\big\}e(-tx)\, \dx \\
	&= \int_0^{\infty} \frac{2\lambda}{\lambda^2 + 4\pi^2 t^2}\, \dmu
		- \int_0^{\infty} \tK\bigl(\lambda, t\bigr)\, \dmu.
\end{split}
\end{align}
\item[(vi)]  If $\ttK(z)$ is an entire function of exponential type at most $\pi$, then 
\begin{equation}\label{Intro13}
\int_{-\infty}^{\infty} \Big|f_{\mu}(x) - \ttK(x)\Big|\, \dx \geq \int_0^{\infty}\big\{\tfrac{2}{\lambda} - \tfrac{2}{\lambda}\sech\bigl(\tfrac{\lambda}{2}\bigr)\big\}\, \dmu.
\end{equation}
\item[(vii)]  There is equality in the inequality {\rm (\ref{Intro13})} if and only if $\ttK(z) = K_{\mu}(z)$.
\end{itemize}

\end{theorem}
Theorem \ref{thm1.2} was stated for the best approximation of exponential type at most $\pi$ of $f_{\mu}(x)$. It is often useful to have results of the same sort in which the entire approximations have exponential type at most $\pi\delta$, where $\delta$ is a positive parameter. To accomplish this we introduce a second measure $\nu$ defined on Borel subsets $E\subseteq (0,\infty)$ by
\begin{equation}\label{Intro14}
\nu(E) = \mu(\delta E),
\end{equation}
where
\begin{equation*}
\delta E = \{\delta x: x\in E\}
\end{equation*}
is the dilation of $E$ by $\delta$.  If $\mu$ satisfies (\ref{Intro4}) then $\nu$ also satisfies (\ref{Intro4}), and the two functions 
$f_{\mu}(x)$ and $f_{\nu}(x)$ are related by the identity
\begin{align}\label{Intro15}
\begin{split}
f_{\nu}(x) &= \int_0^{\infty} \big\{e^{-\lambda |x|} - e^{-\lambda}\big\}\, \dnu \\
	   &= \int_0^{\infty} \big\{e^{-\lambda \delta^{-1}|x|} - e^{-\lambda \delta^{-1}}\big\}\, \dmu \\
	   &= \int_0^{\infty} \big\{e^{-\lambda |\delta^{-1}x|} - e^{-\lambda}\big\}\, \dmu
	   		- \int_0^{\infty} \big\{e^{-\lambda \delta^{-1}} - e^{-\lambda}\big\}\, \dmu \\
	   &= f_{\mu}\bigl(\delta^{-1} x\bigr) - f_{\mu}\bigl(\delta^{-1}\bigr).
\end{split}
\end{align}
We apply Theorem \ref{thm1.2} to the functions $f_{\nu}(x)$ and $K_{\nu}(z)$.  Then using (\ref{Intro15}) we obtain
corresponding results for the functions 
\begin{equation*}
f_{\mu}(x) - f_{\mu}\bigl(\delta^{-1}\bigr) = f_{\nu}(\delta x)\quad\text{and}\quad K_{\nu}(\delta z),
\end{equation*}
where the entire function $z\mapsto K_{\nu}(\delta z)$ has exponential type at most $\pi\delta$.  This leads easily to the 
following more general form of Theorem \ref{thm1.2}.

\begin{theorem}\label{thm1.3}  Assume that the measure $\mu$ satisfies  {\rm (\ref{Intro4})}, and let $\nu$ be
the measure defined by {\rm (\ref{Intro14})}, where $\delta$ is a positive parameter.  
\begin{itemize}
\item[(i)]  The real entire function $z\mapsto K_{\nu}(\delta z) + f_{\mu}\bigl(\delta^{-1}\bigr)$ has exponential type at most $\pi\delta$.
\item[(ii)]  For real $x\not= 0$ the function
\begin{equation}\label{Intro16}
\lambda \mapsto e^{-\lambda |x|} - K\bigl(\delta^{-1}\lambda, \delta x\bigr)
\end{equation}
is integrable on $(0, \infty)$ with respect to $\mu$.  
\item[(iii)]  For all real $x$ we have
\begin{align}\label{Intro17}
\begin{split}
f_{\mu}(x) - f_{\mu}\bigl(\delta^{-1}\bigr)&- K_{\nu}(\delta x) \\
	&= \int_0^{\infty} \big\{e^{-\lambda |x|} - K\bigl(\delta^{-1}\lambda, \delta x\bigr)\big\}\, \dmu.
\end{split}
\end{align}
\item[(iv)]  The function $x\mapsto f_{\mu}(x) - f_{\mu}\bigl(\delta^{-1}\bigr) - K_{\nu}(\delta x)$ 
is integrable on $\R$, and
\begin{align}\label{Intro18}
\begin{split}
\int_{-\infty}^{\infty} \big| f_{\mu}(x) - f_{\mu}\bigl(\delta^{-1}\bigr)&- K_{\nu}(\delta x)\big|\, \dx \\
     &= \int_0^{\infty}\big\{\tfrac{2}{\lambda} - \tfrac{2}{\lambda}\,\sech\left(\tfrac{\lambda}{2\delta}\right) \big\}\, \dmu.
\end{split}
\end{align}
\item[(v)]  If $t\not= 0$ then
\begin{align}\label{Intro19}
\begin{split}
\int_{-\infty}^{\infty} \big\{&f_{\mu}(x) - f_{\mu}\bigl(\delta^{-1}\bigr) - K_{\nu}(\delta x)\big\}e(-tx)\, \dx \\
	&= \int_0^{\infty} \frac{2\lambda}{\lambda^2 + 4\pi^2 t^2}\, \dmu
		- \delta^{-1} \int_0^{\infty} \tK\bigl(\delta^{-1}\lambda, \delta^{-1}t\bigr)\, \dmu.
\end{split}
\end{align}
\item[(vi)]  If $\ttK(z)$ is an entire function of exponential type at most $\pi\delta$, then 
\begin{equation}\label{Intro20}
\int_{-\infty}^{\infty} \left|f_{\mu}(x) - \ttK(x)\right|\dx \geq \int_0^{\infty}\big\{\tfrac{2}{\lambda} - \tfrac{2}{\lambda}\,\sech\left(\tfrac{\lambda}{2\delta}\right)\big\} \, \dmu.
\end{equation}
\item[(vii)]  There is equality in the inequality {\rm (\ref{Intro20})} if and only if $\ttK(z) = K_{\nu}(\delta z) + f_{\mu}\bigl(\delta^{-1}\bigr)$.
\end{itemize}
\end{theorem}  

To illustrate how these results can be applied, we consider the problem of approximating the function 
$x\mapsto \log |x|$ by an entire function $z\mapsto V(z)$ of exponential type at most $\pi$. We select $\mu$ to be a Haar measure on the multiplicative group $(0, \infty)$, so that 
\begin{equation}\label{Intro21}
\mu(E) = \int_E \lambda^{-1}\ \dl
\end{equation}
for all Borel subsets $E \subseteq(0,\infty)$.  For this measure $\mu$ we find that
\begin{equation*}\label{Intro22}
f_{\mu}(x) = - \log |x|.
\end{equation*}
We apply Theorem \ref{thm1.2} with $V(z) = - K_{\mu}(z)$, that is
\begin{align}\label{Intro23}
\begin{split}
V(z) = \lim_{N\rightarrow\infty} \left(\frac{\cos \pi z}{\pi}\right)
			\Bigg\{\sum_{n=-N}^{N+1}&\frac{(-1)^{n}\log |n-\hh|}{(z - n + \hh)} \Bigg\},
\end{split}
\end{align}
where the limit converges uniformly on compact subsets of $\C$.  From Theorem \ref{thm1.2} we conclude that 
$V(z)$ is the best approximation of exponential type at most $\pi$ for $\log|x|$ with 
\begin{equation}\label{Intro24}
\int_{-\infty}^{\infty} \left| \log |x|- V(x)\right| \dx = \frac{4}{\pi}\sum_{n=0}^{\infty}\frac{(-1)^{n}}{(2n+1)^2} = \frac{4G}{\pi}\,,
\end{equation}
where $G = 0.915965594...$ is the Catalan's constant. This follows from (\ref{Intro11}) and standard contour integration.\\

In a similar manner, Theorem \ref{thm1.3} can be applied to determine the entire function
of exponential type at most $\pi\delta$ that best approximates $x\mapsto \log |x|$.  Alternatively, the functional
equation for the logarithm allows us to accomplish this directly.  Clearly the entire function
\begin{equation*}
z\mapsto -\log \delta + V(\delta z)
\end{equation*}
has exponential type at most $\pi\delta$ and is the best approximation to $x\mapsto \log |x|$ on $\R$, satisfying
\begin{equation}\label{Intro25}
\int_{-\infty}^{\infty} \big|\log|x| +\log \delta - V(\delta x) \big|\,\dx = \frac{4G}{\pi\delta}.
\end{equation}
Another interesting application of Theorem \ref{thm1.2} arises when we choose measures $\mu_{\sigma}$ such that 
\begin{equation}\label{Intro26}
\mu_{\sigma}(E) = \int_E \lambda^{-\sigma}\ \dl,
\end{equation}
for all Borel subsets $E\subseteq (0,\infty)$.  For $0 < \sigma < 2$ the measure $\mu_{\sigma}$ satisfies the 
condition (\ref{Intro4}). We find that
\begin{align}\label{Intro27}
\begin{split}
f_{\mu_{\sigma}}(x) & = \int_0^{\infty} \big\{e^{-\lambda |x|} - e^{-\lambda}\big\}\ \lambda^{-\sigma}\ \dl \\
                    & = \Gamma(1 - \sigma)\big\{|x|^{\sigma - 1} - 1\big\}\, ,\quad\text{if $\sigma\not= 1$.}
\end{split}
\end{align}
Therefore, if we want to find the best approximation of exponential type at most $\pi$ for the even function $x \mapsto |x|^{\sigma-1}$ where $0< \sigma < 2$ and $\sigma \neq 1$, we should consider
\begin{equation*}
V_{\sigma}(z) = \frac{K_{\mu_{\sigma}}(z)}{\Gamma(1-\sigma)} + 1 = \lim_{N\rightarrow\infty} \left(\frac{\cos \pi z}{\pi}\right)
			\Bigg\{\sum_{n=-N}^{N+1}\frac{(-1)^{n}|n-\hh|^{\sigma-1}}{(z - n + \hh)} \Bigg\}.
\end{equation*}
From (\ref{Intro11}) and contour integration we conclude that 
\begin{align}
\begin{split}
\int_{-\infty}^{\infty} \left| |x|^{\sigma-1} - V_{\sigma}(x) \right|\, \dx &= \frac{1}{\Gamma(1 - \sigma)}\int_0^{\infty}\big\{\tfrac{2}{\lambda} - \tfrac{2}{\lambda}\sech\bigl(\tfrac{\lambda}{2}\bigr)\big\}\ \lambda^{-\sigma}\ \dl \\
& = \frac{1}{\Gamma(1 - \sigma)}\ \frac{4}{\sin\left(\frac{\pi \sigma}{2}\right)\pi^{\sigma}}\  \sum_{n = 0}^{\infty}\frac{(-1)^n}{(2n+1)^{1+\sigma}}.
\end{split}
\end{align}
Our results can also be used to approximate certain real valued periodic functions by trigonometric 
polynomials.  This is accomplished by applying the Poisson summation formula to the functions $x \mapsto e^{-\lambda|x|}$ and $x \mapsto K(\lambda,x)$, and then integrating the parameter $\lambda$ with respect to a measure $\mu$.
We give a general account of this method in section 6. An interesting special case of Theorem \ref{thm6.2} occurs when we consider $\mu$ to be the Haar measure defined in (\ref{Intro21}). In this case, we obtain the trigonometric polynomial of degree $N$ that best approximates in $L^1(\R/\Z)$ the function $x \mapsto \log|1 - e(x)|$. Here is the precise result.

\begin{theorem}
 Let $N$ be a nonnegative integer. There exists a real valued trigonometric polynomial
\begin{equation}
 v_N(x) = \sum_{n=-N}^{N} \widehat{v}_N(n)e(nx)
\end{equation}
that is the best approximation in $L^1(\R/\Z)$ for the function $x \mapsto \log|1 - e(x)|$. Precisely, if $\wv(x)$ is a trigonometric polynomial of degree at most $N$, we have
\begin{equation}
 \int_{\R/\Z} \left| \log|1 - e(x)| - \wv(x) \right|\, \dx \geq \frac{4G}{(2N+2)\pi}\,,
\end{equation}
with equality if and only if $\wv(x) = v_N(x)$. Here $G = 0.915965594...$ is the Catalan's constant.
\end{theorem}
The trigonometric polynomial $v_N(x)$ is explicitly described in section 6, equations (\ref{ef60})-(\ref{ef62}). With the notation of section 6 we have $v_N(x) = - k_{\mu}(N;x)$, for this particular measure $\mu$.

\section{Proof of Theorem \ref{thm1.1}}

By  performing a change of variables, it suffices to prove (\ref{Intro2}) for $\delta = 1$ and all $\lambda >0$. We start by defining the following entire function of exponential type $\pi$
\begin{equation*}
A(\lambda, z) = \left( \frac{\sin \pi z}{\pi} \right)\sum_{n=0}^{\infty} (-1)^n\frac{ e^{-\lambda n}}{(z-n)}.
\end{equation*}
We also define the function $B:\R \to \R$ by
\begin{equation*}
B(w) = - \frac{e^w}{e^w + 1}.
\end{equation*}
\begin{lemma}
If $\Re(z) < 0$ we have
\begin{equation}\label{Sec2.1}
A(\lambda, z) = \left( \frac{\sin \pi z}{\pi} \right) \int_0^{\infty} B(\lambda + w) \, e^{z w} \, \dw\, ,
\end{equation}
and if $\Re(z) >0$ we have
\begin{equation}\label{S2.1.1}
A(\lambda, z) = e^{-\lambda z} - \left( \frac{\sin \pi z}{\pi} \right) \int_{-\infty}^{0} B(\lambda + w) \, e^{z w} \, \dw.
\end{equation}
\end{lemma}
\begin{proof}
Let $\rho>0$. If $\Re(z) \leq -\rho$, then
\begin{eqnarray*}
\int_0^{\infty} B(\lambda + w) \, e^{z w} \, \dw  &=&  e^{-\lambda z}\int_{\lambda}^{\infty} B(w) e^{zw} \dw \\
 &=&  e^{-\lambda z}\int_{\lambda}^{\infty} \sum_{n=0}^{\infty} (-1)^{n+1}e^{(z-n)w} \dw .
\end{eqnarray*}
Now
\begin{equation*}
\left|\sum_{n=0}^{\infty} (-1)^{n+1}e^{(z-n)w}\right| \leq \sum_{n=0}^{\infty} e^{-\rho w - n w} = \left(\frac{e^w}{e^w - 1}\right)e^{-\rho w}\,,
\end{equation*}
so by the dominated convergence theorem we have
\begin{eqnarray*}
\int_0^{\infty} B(\lambda + w) \, e^{z w} \, \dw  &=&  e^{-\lambda z} \sum_{n=0}^{\infty} (-1)^{n+1} \int_{\lambda}^{\infty}e^{(z-n)w}\ \dw \\
 &=&  \sum_{n=0}^{\infty} (-1)^{n} \frac{e^{-\lambda n}}{(z-n)}
\end{eqnarray*}
and this proves (\ref{Sec2.1}).\\

Suppose now that $\Re(z) \geq \rho >0$. Then
\begin{equation}\label{S2.1}
\int_{-\infty}^{0} B(\lambda + w) \, e^{z w} \, \dw = e^{-\lambda z} \left\{ \int_{-\infty}^0 B(w) e^{zw} \dw + \int_0^{\lambda} B(w) e^{zw} \dw \right\}.
\end{equation}
The first of these integrals is equal to 
\begin{eqnarray}\label{S2.2}
\int_{-\infty}^0 B(w) e^{zw} \dw &=& \int_{-\infty}^0 \left( \sum_{n=1}^{\infty} -e^{(2n-1)w} + e^{2n w} \right) e^{zw} \dw.
\end{eqnarray}
For $w<0$ we have
\begin{equation*}
\left|\left(\sum_{n=1}^{\infty} -e^{(2n-1)w} + e^{2n w} \right) e^{zw}\right| \leq -B(w)e^{\rho w}\, ,
\end{equation*}
and therefore we can use the dominated convergence theorem to conclude that (\ref{S2.2}) is equal to 
\begin{align}\label{S2.2.1}
\begin{split}
\int_{-\infty}^0 B(w) e^{zw} \dw &= \sum_{n=1}^{\infty}\left( \int_{-\infty}^0 -e^{(z + 2n-1)w} + e^{(z + 2n) w}\ \dw \right) \\
& = \sum_{n=1}^{\infty} \left( -\frac{1}{(z+2n-1)} + \frac{1}{(z+2n)} \right) = \sum_{n=1}^{\infty} \frac{(-1)^n}{(z+n)}.
\end{split}
\end{align}
Analogously, for the second integral in (\ref{S2.1}) we have
\begin{align}\label{S2.3}
\begin{split}
\int_0^{\lambda}  B(w)\, e^{zw}\,  \dw  &= \int_0^{\lambda} \left( \sum_{n=0}^{\infty} -e^{-2n w} + e^{-(2n+1) w} \right) e^{zw}\  \dw \\
& =  \sum_{n=0}^{\infty} \left(\int_0^{\lambda} -e^{(z -2n)w} + e^{(z - 2n -1) w}\ \dw \right) \\
& =  \sum_{n=0}^{\infty} \frac{(-1)^n}{(z-n)} + \sum_{n=0}^{\infty} (-1)^{n+1}\frac{e^{(z-n)\lambda}}{(z-n)}.
\end{split}
\end{align}
Putting together (\ref{S2.2.1}) and (\ref{S2.3}) in expression (\ref{S2.1}), and using the identity
\begin{equation*}
\frac{\pi}{\sin \pi z} = \sum_{n \in \Z} \frac{(-1)^n}{(z-n)}\,,
\end{equation*}
we conclude the proof of (\ref{S2.1.1}).
\end{proof}

We now proceed to the proof of (\ref{Intro3}). As the function $x \mapsto K(\lambda,x)$ is even, it suffices to prove (\ref{Intro3}) assuming $x\geq 0 $. We first observe that 
\begin{equation}\label{S2.4}
K(\lambda, z) = e^{-\tfrac{\lambda}{2}} \left\{ A\left(\lambda, z - \tfrac{1}{2} \right) + A\left(\lambda, -z - \tfrac{1}{2} \right) \right\}.
\end{equation}
Note that the right-hand side of (\ref{S2.1.1}) defines an analytic function for $\Re(z) > -1$, and this implies that (\ref{S2.1.1}) is true for $\Re(z) > -1$ by analytic continuation. If $x\geq 0$, then $x-\h > -1$ and equation (\ref{S2.1.1}) gives us
\begin{equation}\label{S2.5}
A\left(\lambda, x - \tfrac{1}{2} \right) = e^{-\lambda \left(x - \tfrac{1}{2}\right)} + \left(\frac{\cos \pi x}{\pi}\right) \int_{-\infty}^0 B(\lambda + w) e^{xw - w/2}\  \dw ,
\end{equation}
and as we have $-x-\h <0$, equation (\ref{Sec2.1}) gives us
\begin{equation}\label{S2.6}
A\left(\lambda, -x - \tfrac{1}{2} \right) =  - \left(\frac{\cos \pi x}{\pi}\right) \int_{0}^{\infty} B(\lambda + w) e^{-xw - w/2}\  \dw.
\end{equation}
We define the function $C(w) = B(w)e^{-w/2}$, and use (\ref{S2.5}) and (\ref{S2.6}) in expression (\ref{S2.4}) to obtain
\begin{equation}\label{S2.7}
e^{-\lambda x} - K(\lambda ,x) = \left(\frac{\cos \pi x}{\pi}\right)  \int_0^{\infty} \bigl\{ C(\lambda + w)  - C(\lambda - w) \bigr\} e^{-x w}\  \dw.
\end{equation}
Now it is just a matter of observing that 
$$ C(w) = - \frac{1}{e^{w/2} + e^{-w/2}}$$
is an even function, which is strictly increasing for $w>0$. Therefore, for $\lambda >0$ and $w>0$, we have
$$C(\lambda - w) = C(|\lambda - w|) < C(\lambda + w)\,,$$
and the integral in (\ref{S2.7}) above is strictly positive. This proves that the sign of $e^{-\lambda |x|} - K(\lambda ,x)$ is the same as the sign of $\cos \pi x$, which is part (ii) of Theorem \ref{thm1.1}.\\

To prove part (i), we first verify that $x \mapsto e^{-\lambda |x|} - K(\lambda ,x)$ is integrable. In fact,
\begin{align}\label{S2.8}
\begin{split}
\int_{-\infty}^{\infty} \big|e^{-\lambda |x|} & - K(\lambda ,x)\big|\,\dx  \,= \, 2 \int_{0}^{\infty} \big|e^{-\lambda x} - K(\lambda ,x)\big|\,\dx \\
& =  2 \int_0^{\infty} \left|\frac{\cos \pi x}{\pi}\right| \int_0^{\infty} \bigl\{ C(\lambda + w)  - C(\lambda - w) \bigr\} e^{-x w} \dw \ \dx \\
& =  \int_0^{\infty} \bigl\{ C(\lambda + w)  - C(\lambda - w) \bigr\} \int_0^{\infty} \left|\frac{\cos \pi x}{\pi}\right| e^{-xw} \dx \ \dw\\
& \leq  \int_0^{\infty} \frac{C(\lambda + w)  - C(\lambda - w)}{\pi w} \, \dw < \infty.
\end{split}
\end{align}
Let $\ttK(z)$ be a function of exponential type at most $\pi$ such that $x \mapsto e^{-\lambda|x|} - \ttK(x) $ is integrable. This implies that $\ttK(x)$ is integrable. From a classical result of Polya and Plancherel (see equation (\ref{ef1}) in section 6), the function $\ttK(x)$ is bounded on $\R$. We write
$$\psi(x) = e^{-\lambda |x|}  - \ttK(x).$$
From Paley-Wiener theorem, the Fourier transform of $\ttK(x)$ is supported on the interval $[-\h,\h]$. Therefore,
\begin{equation}\label{S2.9}
\widehat{\psi}(t) = \frac{2\lambda}{\lambda^2 + 4\pi^2 t^2} \ \ \ \textrm{for} \ \ \ |t| \geq \h.
\end{equation}
The function $\sgn(\cos \pi x)$ has period $2$ and Fourier series expansion
\begin{equation}\label{S2.10}
\sgn(\cos \pi x) = \frac{2}{\pi}\sum_{k=-\infty}^{\infty} \frac{(-1)^k}{(2k+1)}\ e\left((k + \hh) x \right).
\end{equation}
As $\sgn(\cos \pi x)$ is a normalized function of bounded variation on $[0,2]$, this Fourier expansion converges at every point $x$ and the partial sums are uniformly bounded. Using (\ref{S2.9}) and (\ref{S2.10}) we obtain the lower bound

\begin{align}\label{S2.11}
\begin{split}
\int_{-\infty}^{\infty} \left| e^{-\lambda |x|}  - \ttK(x) \right| \dx  & \geq \left| \int_{-\infty}^{\infty} \psi(x) \sgn(\cos \pi x)\, \dx \right|\\
& = \left| \frac{2}{\pi}\sum_{k=-\infty}^{\infty} \frac{(-1)^k}{(2k+1)}\, \int_{-\infty}^{\infty} \psi(x) e\left((k + \hh) x \right) \dx \right|\\
& = \left| \frac{2}{\pi}\sum_{k=-\infty}^{\infty} \frac{(-1)^k}{(2k+1)}\ \widehat{\psi}\left(-(k + \hh)\right) \right|\\
& = \left| \frac{2}{\pi}\sum_{k=-\infty}^{\infty} \frac{(-1)^k}{(2k+1)}\ \frac{2\lambda}{\left(\lambda^2 + 4\pi^2(k+\hh)^2\right)} \right|\\
& = \frac{2}{\lambda} - \frac{2}{\lambda}\,\sech\left(\frac{\lambda}{2}\right).
\end{split}
\end{align}
The last sum in (\ref{S2.11}) can be calculated by integrating the meromorphic function
\begin{equation*}
H(z) = \frac{1}{z \cos \pi z}\left(\frac{2\lambda}{\lambda^2 + 4 \pi^2 z^2 }\right)
\end{equation*}
along the positively oriented square contour connecting the vertices $-N-Ni$, $N-Ni$, $N+Ni$ and $-N +Ni$, where $N$ is a natural number with $N\to \infty$.\\

From part (ii) of Theorem \ref{thm1.1} that we already proved, it is clear that equality occurs in (\ref{S2.11}) if $\ttK(z) = K(\lambda,z)$. On the other hand, if we assume that there is equality in (\ref{S2.11}) then $\psi(x) \sgn(\cos \pi x)$ does not change sign (both in its real and imaginary parts). Since $\ttK(x)$ is continuous, we deduce that 
$$\ttK(n-\hh) = e^{-\lambda\left|n - \hh\right|}$$
for all $n \in \Z$. From classical interpolation formulas (see \cite[vol II, p.275]{Z} or \cite[p.187]{V}) we conclude that 
$$\ttK(z) = K(\lambda,z) + \beta \cos(\pi z)$$
for some constant $\beta$. But we have seen that $\ttK(x)$ and $K(\lambda,x)$ are integrable, thus $\beta = 0$. This concludes the proof of Theorem \ref{thm1.1}.

\section{Growth estimates in the complex plane}
Let $\rr = \{z\in\C: 0<\Re(z)\}$ denote the open right half plane.  Throughout this section we work with a function 
$\Phi(z)$ that is analytic on $\rr$ and satisfies the following conditions: If $0 < a < b < \infty$ then
\begin{equation}\label{cv10}
\lim_{y\rightarrow \pm\infty} e^{-\pi |y|} \int_a^b \left|\frac{\Phi(x+iy)}{x+iy}\right|\, \dx = 0,
\end{equation}
if $0 < \eta < \infty$ then
\begin{equation}\label{cv11}
\sup_{\eta \le x} \int_{-\infty}^{\infty} \left|\frac{\Phi(x+iy)}{x+iy}\right| e^{-\pi |y|}\ \dy < \infty,
\end{equation}
and
\begin{equation}\label{cv12}
\lim_{x\rightarrow \infty} \int_{-\infty}^{\infty} \left|\frac{\Phi(x+iy)}{x+iy}\right| e^{-\pi |y|}\ \dy = 0.
\end{equation}

\begin{lemma}\label{lem2.1}  Assume that the analytic function $\Phi:\rr\rightarrow \C$ satisfies the 
conditions {\rm (\ref{cv10})}, {\rm (\ref{cv11})}, and {\rm (\ref{cv12})}, and let $0 < \delta$.  Then there 
exists a positive number $c(\delta,\Phi)$, depending only on $\delta$ and $\Phi$, such that the inequality
\begin{equation}\label{cv13}
|\Phi(z)| \le c(\delta,\Phi) |z| e^{\pi |y|}
\end{equation}
holds for all $z = x+iy$ in the closed half plane $\{z\in\C: \delta\le \Re(z)\}$.
\end{lemma}

\begin{proof}  Write $\eta = \min\{\frac14, \h \delta\}$, and set
\begin{equation*}
c_1(\eta, \Phi) = \sup\left\{\int_{-\infty}^{\infty}
			\left|\frac{\Phi(u+iv)}{u+iv}\right|e^{-\pi|v|}\ \dv : \eta \le u \right\}.
\end{equation*}
Then $c_1(\eta, \Phi)$ is finite by (\ref{cv11}).  Let $z = x+iy$ satisfy
$\delta \le \Re(z)$ and let $T$ be a positive real parameter such that $|y|+ \eta < T$.  Then write 
$\Gamma(z, \eta, T)$ for the simply connected, positively oriented, rectangular path connecting the
points $x-\eta - iT, x+\eta - iT, x+\eta + iT, x-\eta + iT,$ 
and $x-\eta - iT$.  From Cauchy's integral formula we have
\begin{equation}\label{cv14}
\frac{\Phi(z)}{z} = \frac{1}{2\pi i}\int_{\Gamma(z, \eta, T)}
			\frac{\Phi(w)}{w(w - z) \cos \pi (w-z)}\ \dw.
\end{equation}
At each point $w = u + iv$ on the path $\Gamma(z, \eta, T)$ we find that
\begin{equation}\label{cv15}
\eta \le |w - z|
\end{equation}
and 
\begin{align*}
\begin{split}
\frac{1}{|\cos \pi (w-z)|^2} &= \frac{2}{\bigl(\cos 2\pi (u-x) + \cosh 2\pi (v-y)\bigr)} \\
	&\le \frac{2}{\bigl(\cosh 2\pi (v-y)\bigr)} \\
	&\le 4e^{-2\pi |v-y|} \le 4e^{2\pi (|y| - |v|)},
\end{split}
\end{align*}
which implies
\begin{align}\label{cv16}
\begin{split}
\frac{1}{|\cos \pi (w-z)|} & \le 2e^{\pi (|y| - |v|)}.
\end{split}
\end{align}
Using the estimates (\ref{cv15}) and (\ref{cv16}), together with (\ref{cv10}) we get
\begin{align}\label{cv17}
\begin{split}
\limsup_{T\rightarrow \infty}\ \Bigg|\int_{x-\eta\pm iT}^{x+\eta\pm iT}
                        &\frac{\Phi(w)}{w(w - z)\cos \pi (w-z)}\ \dw\Bigg| \\
	&\le \limsup_{T\rightarrow \infty} 2 \eta^{-1} e^{\pi (|y| - T)} 
			 \int_{x-\eta}^{x+\eta}\left|\frac{\Phi(u\pm iT)}{u\pm iT}\right|\, \du \\
	&= 0.
\end{split}
\end{align}
It follows from (\ref{cv14}) and (\ref{cv17}) that
\begin{align}\label{cv18}
\begin{split}
\frac{\Phi(z)}{z} &= \frac{1}{2\pi i}\int_{x+\eta-i\infty}^{x+\eta+i\infty} 
				\frac{\Phi(w)}{w(w - z)\cos \pi (w-z)}\ \dw \\
	&\qquad\qquad -\frac{1}{2\pi i} \int_{x-\eta-i\infty}^{x-\eta+i\infty}
				\frac{\Phi(w)}{w(w - z)\cos \pi (w-z)}\ \dw.
\end{split}
\end{align}
By appealing to (\ref{cv15}) and (\ref{cv16}) again we find that
\begin{align}\label{cv19}
\begin{split}
\Bigg|\int_{x\pm \eta-i\infty}^{x\pm \eta+i\infty}
				&\frac{\Phi(w)}{w(w - z)\cos \pi (w-z)}\ \dw\Bigg|  \\
	&\le 2 \eta^{-1} e^{\pi |y|}\int_{-\infty}^{\infty} 
			\left|\frac{\Phi(x\pm \eta +iv)}{x\pm \eta +iv}\right|e^{-\pi |v|}\ \dv \\
	&\le 2 c_1(\eta, \Phi)\eta^{-1} e^{\pi |y|}.
\end{split}
\end{align}
Combining (\ref{cv18}) and (\ref{cv19}) leads to the estimate
\begin{equation*}\label{cv20}
\left|\frac{\Phi(z)}{z}\right| \le 2(\pi \eta)^{-1} c_1(\eta, \Phi) e^{\pi |y|},
\end{equation*}
and this plainly verifies (\ref{cv13}) with $c(\delta,\Phi) = 2(\pi \eta)^{-1}c_1(\eta,\Phi)$.
\end{proof}

Let $w=u+iv$ be a complex variable.  From (\ref{cv11}) we find that for each positive real number 
$\beta$ such that $\beta - \h$ is not an integer, and each complex number $z$ with $|\Re(z)|\not=\beta$, the function
\begin{equation*}\label{cv21}
w \mapsto \Bigl(\frac{\cos \pi z}{\cos \pi w}\Bigr)\Bigl(\frac{2w}{z^2 - w^2}\Bigr)\Phi(w)
\end{equation*}
is integrable along the vertical line $\Re(w) = \beta$.  We define a 
complex valued function $z\mapsto I(\beta, \Phi; z)$ on the open set
\begin{equation}\label{open}
\{z\in \C: \left|\Re(z)\right| \neq \beta\}
\end{equation}
by
\begin{equation}\label{cv22}
I(\beta, \Phi; z) = \frac{1}{2\pi i}\int_{\beta-i\infty}^{\beta+i\infty}\Bigl(\frac{\cos \pi z}{\cos \pi w}\Bigr)
				\Bigl(\frac{2w}{z^2 - w^2}\Bigr)\Phi(w)\, \dw.
\end{equation}
It follows using Morera's theorem that $z \mapsto I(\beta, \Phi; z)$ is analytic in each of the three connected components. 

Next we prove a simple estimate for $I(\beta, \Phi; z)$.

\begin{lemma}\label{lem2.2}
Assume that the analytic function $\Phi:\rr \rightarrow \C$ satisfies the conditions
{\rm (\ref{cv10})}, {\rm (\ref{cv11})}, and {\rm (\ref{cv12})}.  Let $\beta$ be a positive real number,
$z = x + iy$ a complex number such that $|\Re(z)|\not=\beta$, and write
\begin{equation}\label{cv25}
B(\beta, \Phi) = \frac{2}{\pi} 
	\int_{-\infty}^{+\infty} \left|\frac{\Phi(\beta + iv)}{\beta + iv}\right|e^{-\pi|v|}\ \dv.
\end{equation}
If $\beta - \h$ is not an integer then
\begin{equation}\label{cv26}
|I(\beta, \Phi; z)| \le B(\beta, \Phi) \left|\sec \pi\beta \right|
	\left(1 + \frac{|z|}{\bigl||x|-\beta\bigr|}\right)e^{\pi |y|}.
\end{equation}

\end{lemma}

\begin{proof}  
On the vertical line $\Re(w)=\beta$ we have
\begin{equation*}\label{cv28}
\bigl||x|-\beta \bigr|\le \min\{|z-w|,|z+w|\}
\end{equation*}
and
\begin{equation*}\label{cv29}
|z|\le \hh|z-w|+\hh|z+w|\le \max\{|z-w|,|z+w|\}.
\end{equation*}
Therefore
\begin{align}\label{cv30}
\begin{split}
\Bigl|\frac{w^2}{z^2-w^2}\Bigr| &\le 1 + \Bigl|\frac{z^2}{z^2-w^2}\Bigr| \\
	&= 1 + |z|^2\big(\min\{|z-w|,|z+w|\}\max\{|z-w|,|z+w|\}\bigr)^{-1} \\
	&\le 1 + \frac{|z|}{\bigl||x|-\beta\bigr|}.
\end{split}
\end{align}
On the line $\Re(w)=\beta$ we also use the elementary inequality
\begin{equation}\label{cv31}
|\cos \pi (\beta + iv)|^{-1} \le 2 e^{-\pi |v|} \left|\sec \pi\beta\right|.
\end{equation}
Then we use (\ref{cv30}) and (\ref{cv31}) to estimate the integral on the right of (\ref{cv22}).  The bound
(\ref{cv26}) follows easily.
\end{proof}

For each positive number $\xi$ we define an even rational function $z\mapsto \A(\xi, \Phi; z)$ 
on $\C$ by
\begin{align}\label{cv32}
\begin{split}
\A(\xi, \Phi; z) &= \frac{\Phi(\xi)}{(z - \xi)} -\frac{\Phi(\xi)}{(z + \xi)}.
\end{split}
\end{align}

\begin{lemma}\label{lem2.3}  
Assume that the analytic function $\Phi:\rr \rightarrow \C$ satisfies the conditions
{\rm (\ref{cv10})}, {\rm (\ref{cv11})}, and {\rm (\ref{cv12})}.  Then the sequence of entire functions
\begin{equation}\label{cv33}
\Bigl(\frac{\cos \pi z}{\pi}\Bigr)\sum_{n=1}^N \,(-1)^n \,\A(n - \hh, \Phi; z),\ \text{where}\ N = 1, 2, 3, \dots ,
\end{equation}
converges uniformly on compact subsets of $\C$ as $N\rightarrow \infty$, and therefore
\begin{equation}\label{cv34}
\K(\Phi, z) = \lim_{N\rightarrow \infty}\Bigl(\frac{\cos \pi z}{\pi}\Bigr) \sum_{n=1}^N \,(-1)^n\,\A(n - \hh, \Phi; z)
\end{equation}
defines an entire function.
\end{lemma}

\begin{proof} We assume that $z$ is a complex number in $\rr$ such that $z-\h$ is not an integer.  Then
\begin{equation}\label{cv37}
w\mapsto \Bigl(\frac{\cos \pi z}{\cos \pi w}\Bigr)\Bigl(\frac{2w}{z^2 - w^2}\Bigr)\Phi(w)
\end{equation}
defines a meromorphic function of $w$ on the right half plane $\rr$.  We find that (\ref{cv37}) has a
simple pole at $w = z$ with residue $-\Phi(z)$.  And for each positive integer $n$, (\ref{cv37}) has a 
pole of order at most one at $w = n - \h$ with residue
\begin{equation*}
\Bigl(\frac{\cos \pi z}{\pi}\Bigr)\,(-1)^n\, \A(n - \hh, \Phi; z). 
\end{equation*} 
Plainly (\ref{cv37}) has no other poles in $\rr$.  Let $0 < \beta < \h$, let $N$ be a positive integer, and $T$ a
positive real parameter.  Write $\Gamma(\beta, N, T)$ for the simply connected, positively oriented
rectangular path connecting the points $\beta - iT$, $N - iT$, $N + iT$, $\beta + iT$ and $\beta - iT$.
If $z$ satisfies $\beta < \Re(z) < N$ and $|\Im(z)| < T$, and $z-\h$ is not an 
integer, then from the residue theorem we obtain the identity
\begin{align}\label{cv38}
\begin{split}
\Bigl(\frac{\cos \pi z}{\pi}\Bigr) &\sum_{n=1}^N \, (-1)^n\,\A(n - \hh, \Phi; z) - \Phi(z) \\
	&=\frac{1}{2\pi i}\int_{\Gamma(\beta, N, T)} \Bigl(\frac{\cos \pi z}{\cos \pi w}
		\Bigr)\Bigl(\frac{2w}{z^2 - w^2}\Bigr)\Phi(w)\ \dw.
\end{split}
\end{align}
We let $T\rightarrow \infty$ on the right hand side of (\ref{cv38}), and we use the hypotheses (\ref{cv10})
and (\ref{cv11}).  In this way we conclude that
\begin{equation}\label{cv39}
\Bigl(\frac{\cos \pi z}{\pi}\Bigr) \sum_{n=1}^N\, (-1)^n \A(n - \hh, \Phi; z) - \Phi(z) = I(N, \Phi; z) - I(\beta, \Phi; z).
\end{equation}
Initially (\ref{cv39}) holds for $\beta < \Re(z) < N$ and $z-\h$ not an integer.  However, we have already observed
that both sides of (\ref{cv39}) are analytic in the strip $\{z\in \C: \beta < \Re(z) < N\}$.  Therefore the condition
that $z-\h$ is not an integer can be dropped.

Now let $M < N$ be positive integers.  From (\ref{cv39}) we find that
\begin{equation}\label{cv40}
\Bigl(\frac{\cos \pi z}{\pi}\Bigr) \sum_{n=M+1}^N  (-1)^n \,\A(n - \hh, \Phi; z)= I(N, \Phi; z) - I(M, \Phi; z)
\end{equation}
in the infinite strip $\{z\in \C: \beta < \Re(z) < M\}$.  In fact we have seen that both sides of (\ref{cv40})
are analytic in the infinite strip $\{z\in \C: |\Re(z)| < M\}$.  Therefore the identity (\ref{cv40}) must hold in
this larger domain by analytic continuation.  Let $\J\subseteq \C$ be a compact set and assume that $L$ is an 
integer so large that $\J \subseteq \{z\in \C: 2|z| < L\}$.  From (\ref{cv12}), Lemma \ref{lem2.2}, and 
(\ref{cv40}), it is obvious that the sequence of entire functions (\ref{cv33}), where $L \le N$,
is uniformly Cauchy on $\J$.  This verifies the lemma showing that (\ref{cv34}) 
defines an entire function.
\end{proof}

\begin{lemma}\label{lem2.4}  Assume that the analytic function $\Phi:\rr \rightarrow \C$ satisfies the conditions
{\rm (\ref{cv10})}, {\rm (\ref{cv11})} and {\rm (\ref{cv12})}.  Let the entire function $\K(\Phi, z)$ be
defined by {\rm (\ref{cv34})}.  If $0 < \beta < \h$ then the identity 
\begin{equation}\label{cv41}
\Phi(z) - \K(\Phi, z) =  I(\beta, \Phi; z)
\end{equation} 
holds for all $z$ in the half plane $\{z\in \C: \beta < \Re(z)\}$, and the identity
\begin{equation}\label{cv42}
-\K(\Phi, z) =  I(\beta, \Phi; z)
\end{equation}
holds for all $z$ in the infinite strip $\{z\in \C: |\Re(z)| < \beta\}$.  
\end{lemma}

\begin{proof}  We argue as in the proof of Lemma \ref{lem2.3}, letting $N\rightarrow \infty$ on both sides of
(\ref{cv39}).  Then we use (\ref{cv12}) and Lemma \ref{lem2.2}, and obtain the identity
\begin{equation*}\label{cv45}
\Phi(z) - \K(\Phi, z) = I(\beta, \Phi; z)
\end{equation*}
at each point of the half plane $\{z\in \C: \beta < \Re(z)\}$.  This proves (\ref{cv41}).

Next, we assume that $|\Re(z)| < \beta$.  In this case the residue theorem provides the identity
\begin{align}\label{cv46}
\begin{split}
\Bigl(\frac{\cos \pi z}{\pi}\Bigr) &\sum_{n=1}^N \,(-1)^n\,\A(n - \hh, \Phi; z) \\
	&= \frac{1}{2\pi i}\int_{\Gamma(\beta, N, T)} \Bigl(\frac{\cos \pi z}{\cos \pi w}
		\Bigr)\Bigl(\frac{2w}{z^2 - w^2}\Bigr)\Phi(w)\ \dw.
\end{split}
\end{align}
We let $T\rightarrow \infty$ and argue as before.  In this way (\ref{cv46}) leads to 
\begin{equation}\label{cv47}
\Bigl(\frac{\cos \pi z}{\pi}\Bigr) \sum_{n=1}^N \,(-1)^n\,\A(n - \hh, \Phi; z) = I(N, \Phi; z) - I(\beta, \Phi; z).
\end{equation}
Then we let $N\rightarrow \infty$ on both sides of (\ref{cv47}) and we use (\ref{cv12}) and Lemma \ref{lem2.2} 
again.  We find that
\begin{equation*}
 -\K(\Phi, z) = I(\beta, \Phi; z),
\end{equation*}
and this verifies (\ref{cv42}).
\end{proof}

\begin{corollary}\label{cor2.5} 
Suppose that $\Phi(z) = 1$ is constant on $\rr$.  If $0 < \beta < \h$ then 
\begin{equation}\label{cv48}
I(\beta, 1; z) = 0,
\end{equation}
in the open half plane $\{z\in\C: \beta < \Re(z)\}$.
\end{corollary}

\begin{proof}  We have
\begin{align*}
\begin{split}
\K(1, z) &= \lim_{N\rightarrow \infty}\left(\frac{\cos \pi z}{\pi}\right) \sum_{n=1}^N \,(-1)^n\,\A(n - \hh, 1; z) \\
     &= \lim_{N\rightarrow \infty}\left(\frac{\cos \pi z}{\pi}\right)\sum_{n=-N+1}^{N}\,(-1)^n (z-n+\hh)^{-1} = 1.
\end{split}
\end{align*}
Now the identity (\ref{cv48}) follows from (\ref{cv41}). 
\end{proof}

\begin{lemma}\label{lem2.6}
Assume that the analytic function $\Phi:\rr \rightarrow \C$ satisfies the conditions
{\rm (\ref{cv10})}, {\rm (\ref{cv11})} and {\rm (\ref{cv12})}.  Let the entire function $\K(\Phi, z)$ be defined by {\rm (\ref{cv34})}. Then there exists a positive number $c(\Phi)$, depending only on $\Phi$, such that the inequality
\begin{equation}\label{cv51}
|\K(\Phi, z)|\le c(\Phi)(1 +|z|)e^{\pi |y|},
\end{equation}
holds for all complex numbers $z=x+iy$.  In particular, $\K(\Phi, z)$ is an entire function of exponential type at most $\pi$.
\end{lemma}

\begin{proof}
In the closed half plane $\{z\in \C: \frac14 \le \Re(z)\}$ the identity (\ref{cv41}) implies that
\begin{equation*}
|\K(\Phi, z)| \le |\Phi(z)| + |I(\tfrac{1}{8}, \Phi; z)|.
\end{equation*}
Then an estimate of the form (\ref{cv51}) in this half plane follows from Lemma \ref{lem2.1} and Lemma \ref{lem2.2}.
In the closed infinite strip $\{z\in \C: |\Re(z)| \le \frac14\}$ we have
\begin{equation*}
|\K(\Phi, z)| = |I(\tfrac{3}{8}, \Phi; z)|
\end{equation*}
from the identity (\ref{cv42}).  Plainly an estimate of the form (\ref{cv51}) in this closed infinite
strip follows from Lemma \ref{lem2.2}.  This suffices to prove inequality (\ref{cv51}) for all complex 
$z$, since $\K(\Phi, z)$ is an even function of $z$.
\end{proof}

\section{Fourier expansions}
\begin{lemma}\label{lem3.2}  If $0 < \beta < \h$, then at each point $z$ in the half plane 
$\{z\in\C:\beta < \Re(z)\}$ we have
\begin{equation}\label{cv-1}
e^{-\lambda z}- K(\lambda, z) = \dfrac{1}{2\pi i} \int_{\beta -i \infty}^{\beta + i \infty} 
	\left(\dfrac{\cos \pi z}{\cos \pi w}\right) \left(\dfrac{2w}{z^2 - w^2}\right) e^{-\lambda w}\ \dw.
\end{equation}
\end{lemma}

\begin{proof}  We apply Lemma \ref{lem2.3} with $\Phi(z) = e^{-\lambda z}$.  It follows that
\begin{equation*}
\K(\Phi, z) = K(\lambda, z).
\end{equation*}
The identity (\ref{cv-1}) follows now from Lemma \ref{lem2.4}.
\end{proof}
As $x\mapsto K(\lambda,x)$ is a restriction of a function of exponential type $\pi$, bounded and 
integrable on $\R$, its Fourier transform
\begin{equation}\label{finite3}
\tK(\lambda,t) = \int_{-\infty}^{\infty}K(\lambda,x)e(-tx)\, \dx
\end{equation}
is a continuous function of the real variable $t$ supported on the interval $[-\h,\h]$.  Then 
by Fourier inversion we have the representation
\begin{equation}\label{finite4}
K(\lambda,z) = \int_{-\h}^{\h}\tK(\lambda,t)e(tz)\, \dt
\end{equation}
for all complex $z$.  It will be useful to have more explicit information about the Fourier transform 
of this function.

\begin{lemma}\label{lem3.3}  For $|t| \leq \h$ the Fourier transform {\rm (\ref{finite3})} is given by
\begin{equation}\label{finite5}
\tK(\lambda, t) = \dfrac{\sinh\left(\frac{\lambda}{2}\right) \cos \pi t }{\sinh^2\left(\frac{\lambda}{2}\right) + \sin ^2\pi t}. 
\end{equation}
From {\rm (\ref{finite5})} we conclude that 
\begin{equation}\label{pft1}
\tK(\lambda, t) \geq 0
\end{equation}
for all $t \in \R$.
\end{lemma}
\begin{proof}
The following entire function
\begin{equation*}
H(z) = \frac{\cos \pi z}{\pi (z + \hh)}
\end{equation*}
has exponential type $\pi$ and, when restricted to $\R$, belongs to $L^2(\R)$. By Paley-Wiener theorem we know that its Fourier transform is supported on $[-\hh,\hh]$, being explicitly given by
\begin{equation}\label{ftcos}
\widehat{H}(t) = e^{\pi i t} 
\end{equation}
for $t \in [-\hh,\hh]$. An adaptation of \cite[Theorem 9]{V}, together with (\ref{ftcos}), show that the entire function of exponential type at most $\pi$, integrable on $\R$,
\begin{equation}
K(\lambda,z) = \sum_{n\in \Z} e^{-\lambda \left|n-\hh\right|} \left( \frac{\cos \pi(z-n)}{\pi(z-n+\hh)} \right)
\end{equation}
has a continuous Fourier transform supported on $[-\hh,\hh]$ given by
\begin{equation}
\tK(\lambda, t) = \sum_{n \in \Z} e^{-\lambda \left|n-\hh\right|}\ e^{-2\pi i tn}\  e^{\pi i t}
\end{equation}
for $t \in [-\hh,\hh]$. This leads to (\ref{finite5}).
\end{proof}
\begin{lemma}\label{lem3.4}  Let $\nu$ be a finite measure on the Borel subsets of $(0,\infty)$.
For each complex number $z$ the function $\lambda\mapsto K(\lambda,z)$ is $\nu$-integrable on $(0,\infty)$.  The complex valued function
\begin{equation}\label{finite10}
\KK_{\nu}(z) =\int_{0}^{\infty} K(\lambda,z)\, \dnu
\end{equation}
is an entire function which satisfies the inequality
\begin{equation}\label{finite11}
|\KK_{\nu}(z)| \leq \nu\{(0,\infty)\}e^{\pi|y|}
\end{equation}
for all $z = x + iy$.  In particular, $\KK_{\nu}(z)$ is an entire function of exponential 
type at most $\pi$.
\end{lemma}

\begin{proof}  We apply (\ref{finite4}) and the fact that $0 \le \tK(\lambda,t)$.  We find that
\begin{align}\label{finite12}
\begin{split}
\int_{0}^{\infty} \left|K(\lambda,z)\right|\, \dnu &=  \int_{0}^{\infty} 
		\left|\int_{-\h}^{\h}\tK(\lambda,t)e(tz)\, \dt\right|\, \dnu \\
	&\leq \int_{0}^{\infty} \int_{-\h}^{\h}\tK(\lambda,t) e^{-2\pi ty}\, \dt\, \dnu \\
	&\leq e^{\pi |y|} \int_{0}^{\infty} \int_{-\h}^{\h}\tK(\lambda,t)\, \dt\, \dnu \\ 
 	&=  e^{\pi |y|}\int_{0}^{\infty} K(\lambda,0)\, \dnu.
\end{split}
\end{align}
As $K(\lambda,0) \le 1$ by (\ref{Intro3}), it follows from (\ref{finite12}) that
\begin{equation*}
\int_{0}^{\infty} \left|K(\lambda,z)\right|\, \dnu \leq \nu\{(0,\infty)\}\, e^{\pi |y|}.
\end{equation*}
This shows that $\lambda\mapsto K(\lambda,z)$
is $\nu$-integrable on $(0,\infty)$ and verifies the bound (\ref{finite11}). It follows easily using Morera's theorem that $z\mapsto \KK_{\nu}(z)$ is an entire function.  Then (\ref{finite11}) implies that this entire function has exponential type at 
most $\pi$.
\end{proof}

Let $\nu$ be a finite measure on the Borel subsets of $(0,\infty)$.  It follows that
\begin{equation}\label{def0}
\Psi_{\nu}(z) = \int_0^{\infty} e^{-\lambda z}\, \dnu
\end{equation}
defines a function that is bounded and continuous in the closed half plane $\{z\in\C: 0\le \Re(z)\}$, and analytic 
in the interior of this half plane.  

\begin{lemma}\label{lem3.5}  If $0 < \beta < \h$, then at each point $z$ in the half plane 
$\{z\in\C:\beta < \Re(z)\}$ we have
\begin{equation}\label{cv6}
\Psi_{\nu}(z) - \KK_{\nu}(z) = \dfrac{1}{2\pi i} \int_{\beta - i\infty}^{\beta + i\infty} 
	\left(\dfrac{\cos \pi z}{\cos \pi w}\right)\left(\dfrac{2w}{z^2 - w^2}\right)\Psi_{\nu}(w)\, \dw.
\end{equation}
\end{lemma}

\begin{proof}  We apply (\ref{cv-1}) and Fubini's theorem to get
\begin{align*}
\Psi_{\nu}(z) - &\KK_{\nu}(z) = \\
	     &= \int_{0}^{\infty} \left\{e^{-\lambda z} - K(\lambda, z)\right\}\, \dnu  \\
	     &= \int_{0}^{\infty} \left\{\dfrac{1}{2\pi i} 
	     	\int_{\beta -i \infty}^{\beta + i \infty} \left( \dfrac{\cos \pi z}{\cos \pi w}\right)
		\left( \dfrac{2w}{z^2 - w^2} \right) e^{-\lambda w}\, \dw \right\}\, \dnu \\
	     &= \dfrac{1}{2\pi i} \int_{\beta -i \infty}^{\beta + i \infty} 
		\left( \dfrac{\cos \pi z}{\cos \pi w}\right) \left( \dfrac{2w}{z^2 - w^2} \right) \Psi_{\nu}(w)\, \dw.
\end{align*}
This proves (\ref{cv6}). 
\end{proof}

\section{Proof of Theorem \ref{thm1.2}}

Let $\mu$ be a measure defined on the Borel subsets of $(0,\infty)$ that satisfies (\ref{Intro4}).
Let $z = x+iy$ be a point in the open right half plane $\rr = \{z\in\C: 0<\Re(z)\}$.  Using (\ref{Intro4}) we 
find that
\begin{equation*}\label{pt-1}
\lambda\mapsto e^{-\lambda z} - e^{-\lambda}
\end{equation*}
is integrable on $(0,\infty)$ with respect to $\mu$.  We define $F_{\mu}:\rr\rightarrow \C$ by
\begin{equation}\label{pt0}
F_{\mu}(z) = \int_0^{\infty} \big\{e^{-\lambda z} - e^{-\lambda}\big\}\, \dmu.
\end{equation}
It follows by applying Morera's theorem that $F_{\mu}(z)$ is analytic on $\rr$.  Also, at each point $z$ 
in $\rr$ the derivative of $F_{\mu}$ is given by
\begin{equation}\label{pt1}
F_{\mu}^{\prime}(z) = - \int_0^{\infty} \lambda e^{-\lambda z}\, \dmu.
\end{equation}
Then (\ref{pt1}) leads to the bound
\begin{equation}\label{pt2}
\bigl|F_{\mu}^{\prime}(x + iy)\bigr| \le \int_0^{\infty} \lambda e^{-\lambda x}\, \dmu = \bigl|F_{\mu}^{\prime}(x)\bigr|.
\end{equation}
From (\ref{pt2}) and the dominated convergence theorem we conclude that
\begin{equation}\label{pt3}
\lim_{x\rightarrow \infty} \bigl|F_{\mu}^{\prime}(x + iy)\bigr| = 0
\end{equation}
uniformly in $y$.  Clearly the functions $f_{\mu}(x)$, defined by (\ref{Intro6}), and $F_{\mu}(z)$, defined 
by (\ref{pt0}), satisfy the identities
\begin{equation}\label{pt4}
f_{\mu}(x) = F_{\mu}\bigl(|x|\bigr)\quad\text{and}\quad f_{\mu}^{\prime}(x) = \sgn(x)F_{\mu}^{\prime}\bigl(|x|\bigr)
\end{equation}
for all real $x \neq 0$.
\begin{lemma}\label{lem4.1}  The analytic function $F_{\mu}(z)$ defined by {\rm (\ref{pt0})} satisfies each of the
three conditions {\rm (\ref{cv10})}, {\rm (\ref{cv11})}, and {\rm (\ref{cv12})}.
\end{lemma}

\begin{proof}  Let $0 < \xi \le 1$.  If $\xi \le \Re(z)$, then from (\ref{pt2}) we obtain the inequality
\begin{align*}\label{pt5}
\begin{split}
\bigl|F_{\mu}(z)\bigr| &= \left| \int_1^z F_{\mu}^{\prime}(w)\, \dw\right| \\
	&\le |z - 1| \max\big\{\bigl|F_{\mu}^{\prime}(\theta z + 1 - \theta)\bigr|: 0 \le \theta \le 1\big\} \\
	&\le (|z| + 1) \bigl|F_{\mu}^{\prime}(\xi)\bigr|,
\end{split}
\end{align*}
and therefore
\begin{equation}\label{pt6}
\left|\frac{F_{\mu}(z)}{z}\right|\le (1 + \xi^{-1})\bigl|F_{\mu}^{\prime}(\xi)\bigr|.
\end{equation}
The conditions (\ref{cv10}) and (\ref{cv11}) follow from the bound (\ref{pt6}).  

Now assume that $1 \le x = \Re(z)$.  We have
\begin{align*}\label{pt7}
\begin{split}
\bigl|F_{\mu}(x + iy)\bigr| &= \left|\int_1^x F_{\mu}^{\prime}(u)\, \du + i\int_0^y F_{\mu}^{\prime}(x + iv)\, \dv\right| \\
		 &\le \int_1^x \bigr|F_{\mu}^{\prime}(u)\bigr|\, \du + |y|\bigl|F_{\mu}^{\prime}(x)\bigr|,
\end{split}
\end{align*}
and therefore
\begin{equation}\label{pt8}
\left|\frac{F_{\mu}(x + iy)}{x + iy}\right|
	\le \frac{1}{x}\int_1^x \bigr|F_{\mu}^{\prime}(u)\bigr|\, \du + \bigl|F_{\mu}^{\prime}(x)\bigr|.
\end{equation} 
Then (\ref{pt3}) and (\ref{pt8}) imply that
\begin{equation*}\label{pt9}
\lim_{x\rightarrow \infty} \left|\frac{F_{\mu}(x + iy)}{x+iy}\right| = 0
\end{equation*}
uniformly in $y$.  The remaining condition (\ref{cv12}) follows from this.
\end{proof}

We are now in position to apply the results of sections 3 and 4 to the function $F_{\mu}(z)$.
In view of the identities (\ref{pt4}), the entire function $K_{\mu}(z)$, defined by (\ref{Intro8}), and the entire
function $\K(F_{\mu}, z)$, defined by (\ref{cv34}), are equal.  If $0 < \beta < \h$, and $\beta < \Re(z)$, then 
from (\ref{cv41}) of Lemma \ref{lem2.4} we have
\begin{equation}\label{pt10}
F_{\mu}(z) - K_{\mu}(z) = I(\beta, F_{\mu}; z).
\end{equation}
Applying Lemma \ref{lem2.6} we conclude that $K_{\mu}(z)$ is an entire function of exponential type at most $\pi$.
This verifies (i) in the statement of Theorem \ref{thm1.2}.

Next we define a sequence of measures $\nu_1, \nu_2, \nu_3, \dots $ on Borel subsets $E\subseteq (0,\infty)$ by
\begin{equation}\label{pt20}
\nu_n(E) = \int_E \bigl(e^{-\lambda/n} - e^{-\lambda n}\bigr)\, \dmu,\quad\text{for}\quad n = 1, 2, \dots .
\end{equation}
Then
\begin{align*}
\nu_n\{(0,\infty)\} &= \int_0^{\infty} \int_{1/n}^n \lambda e^{-\lambda u}\, \du\, \dmu \\
	&= - \int_{1/n}^n F_{\mu}^{\prime}(u)\, \du \\
	&= F_{\mu}(1/n) - F_{\mu}(n) < \infty, 
\end{align*}
and therefore $\nu_n$ is a finite measure for each $n$.  It will be convenient to simplify the notation used in (\ref{finite10}) and (\ref{def0}).  For $z$ in $\C$ and $n$ a positive integer we write
\begin{equation}\label{pt21}
K_n(z) =\int_{0}^{\infty} K(\lambda,z)\, \dnnu,
\end{equation} 
and for $z$ in $\rr$ we write
\begin{equation}\label{pt22}
\Psi_n(z) = \int_0^{\infty} e^{-\lambda z}\, \dnnu.
\end{equation}
From Lemma \ref{lem3.4} we learn that $K_n(z)$ is an entire function of exponential type at most $\pi$.
If $0 < \beta < \h$ then (\ref{cv48}) and (\ref{cv6}) imply that
\begin{equation}\label{pt23}
\Psi_n(z) - K_n(z) = I(\beta, \Psi_n; z) = I\bigl(\beta, \Psi_n - \Psi_n(1); z\bigr)
\end{equation}
for all complex $z$ such that $\beta < \Re(z)$. From the definitions (\ref{pt20}), (\ref{pt21}), and (\ref{pt22}), we find that
\begin{equation}\label{pt24}
\Psi_n(x) - K_n(x) 
	= \int_0^{\infty} \bigl(e^{-\lambda x} - K(\lambda, x)\bigr)\bigl(e^{-\lambda/n} - e^{-\lambda n}\bigr)\, \dmu
\end{equation}
for all positive real $x$.   

Let $w = u + iv$ be a point in $\rr$.  Then
\begin{equation}\label{pt26}
\Psi_n(w) - \Psi_n(1) = \int_0^{\infty}\bigl(e^{-\lambda w} - e^{-\lambda}\bigr)
		\bigl(e^{-\lambda/n} - e^{-\lambda n}\bigr)\, \dmu,
\end{equation}
and
\begin{equation*}\label{pt27}
\bigl|e^{-\lambda/n} - e^{-\lambda n}\bigr| \le 1
\end{equation*}
for all positive real $\lambda$ and positive integers $n$.  We let $n\rightarrow \infty$ on both sides
of (\ref{pt26}) and apply the dominated convergence theorem.  In this way we conclude that
\begin{equation}\label{pt28}
\lim_{n\rightarrow \infty} \Psi_n(w) - \Psi_n(1) = F_{\mu}(w)
\end{equation}
at each point $w$ in $\rr$.  If $0 < \beta < \h$ then, as in the 
proof of Lemma \ref{lem4.1}, on the line $\beta = \Re(w)$ we have
\begin{align*}\label{pt28}
\bigl|\Psi_n(w) - \Psi_n(1)\bigr| &\le \int_0^{\infty}\left|\int_1^w \lambda e^{-\lambda t}\ \dt\right|\, \dmu \\
	 &\le (|w| + 1)\bigl|F_{\mu}^{\prime}(\beta)\bigr|.
\end{align*}
It follows that
\begin{equation*}\label{pt30}
\left|\frac{\Psi_n(w) - \Psi_n(1)}{w}\right|
\end{equation*}
is bounded on the line $\beta = \Re(w)$.  From this observation, together with (\ref{pt23}) and (\ref{pt28}),
we conclude that
\begin{align}\label{pt31}
\begin{split}
\lim_{n\rightarrow \infty} \Psi_n(z) - K_n(z) &= \lim_{n\rightarrow \infty} I(\beta, \Psi_n - \Psi_n(1); z) \\ 
	&= I(\beta, F_{\mu}; z) \\ 
	&= F_{\mu}(z) - K_{\mu}(z)
\end{split} 
\end{align}
at each complex number $z$ with $\beta < \Re(z)$.  In particular, we have
\begin{equation}\label{pt32}
\lim_{n\rightarrow \infty} \Psi_n(x) - K_n(x) = F_{\mu}(x) - K_{\mu}(x)
\end{equation}
for all positive $x$.  We combine (\ref{pt24}), (\ref{pt32}) and (\ref{Intro3}) to use the monotone convergence theorem.  This leads to the identity
\begin{equation}\label{pt33}
F_{\mu}(x) - K_{\mu}(x) = \int_0^{\infty} \bigl(e^{-\lambda x} - K(\lambda, x)\bigr)\, \dmu
\end{equation}
for all positive $x$.  Then we use the identity on the left of (\ref{pt4}), and the fact that $x\mapsto K_{\mu}(x)$
is an even function, to write (\ref{pt33}) as
\begin{equation}\label{pt34}  
f_{\mu}(x) - K_{\mu}(x) = \int_0^{\infty} \bigl(e^{-\lambda |x|} - K(\lambda, x)\bigr)\, \dmu
\end{equation}
for all $x\not= 0$.  If $f_{\mu}(0)$ is finite then (\ref{pt34}) holds at $x = 0$ by continuity.  If 
$f_{\mu}(0) = \infty$ then both sides of (\ref{pt34}) are $\infty$. This establishes both (ii) and (iii) in the statement of Theorem \ref{thm1.2}.

Because of (\ref{Intro3}), we get
\begin{align}\label{pt40}
\begin{split}
\int_{-\infty}^{\infty} \left|f_{\mu}(x) - K_{\mu}(x)\right|\, \dx
 	&= \int_{-\infty}^{\infty} \int_0^{\infty} \bigl|e^{-\lambda |x|} - K(\lambda, x)\bigr|\, \dmu\, \dx \\
	&= \int_0^{\infty} \int_{-\infty}^{\infty} \bigl|e^{-\lambda |x|} - K(\lambda, x)\bigr|\, \dx\, \dmu \\
	&= \int_0^{\infty} \big\{\tfrac{2}{\lambda} - \tfrac{2}{\lambda}\sech\left(\tfrac{\lambda}{2}\right)\bigr\}\, \dmu
\end{split}
\end{align}
by Fubini's theorem.  This proves (iv) of Theorem \ref{thm1.2}. Similarly, if $t\not= 0$ we find that
\begin{align}\label{pt41}
\begin{split}
\int_{-\infty}^{\infty} &\big\{f_{\mu}(x) - K_{\mu}(x)\big\}e(-tx)\, \dx \\
	&= \int_{-\infty}^{\infty} \Bigg\{\int_0^{\infty} \bigl(e^{-\lambda |x|} 
		- K(\lambda, x)\bigr)\, \dmu\Bigg\}e(-tx) \, \dx \\
	&= \int_0^{\infty} \Bigg\{\int_{-\infty}^{\infty} \bigl(e^{-\lambda |x|} 
		- K(\lambda, x)\bigr)e(-tx) \, \dx\Bigg\}\, \dmu \\
	&= \int_0^{\infty} \frac{2\lambda}{\lambda^2 + 4\pi^2t^2}\ \dmu
	 	- \int_0^{\infty} \tK(\lambda, t)\, \dmu.
\end{split}
\end{align}
This proves (v) in Theorem \ref{thm1.2}.

Finally, we assume that $\ttK(z)$ is an entire function of exponential type at most $\pi$ and that
\begin{equation}\label{pt45}
\int_{-\infty}^{\infty} \Big|f_{\mu}(x) - \ttK(x)\Big|\, \dx < \infty.
\end{equation}
By the triangle inequality $K_{\mu}(x) - \ttK(x)$ is integrable, and since it has exponential type at most $\pi$, we know that its Fourier transform is supported on $[-\hh,\hh]$. Moreover, by a result of Polya and Plancherel (see equation (\ref{ef1}) in section 6) the function $K_{\mu}(x) - \ttK(x)$ is bounded. We write
\begin{equation}\label{pt45.5}
\psi(x) = f_{\mu}(x) - \ttK(x) = \{f_{\mu}(x) - K(x)\} + \{K(x) - \ttK(x)\}.
\end{equation}
From (\ref{pt45.5}) and (\ref{pt41}) we conclude that the Fourier transform of $\psi(x)$ is given by
\begin{equation}
\widehat{\psi}(t) = \int_0^{\infty} \Big\{\frac{2\lambda}{\lambda^2 + 4\pi^2t^2}\Big\}\, \dmu \ \ \ \textrm{for} \ \ \ |t|\geq \hh.
\end{equation}
We simply proceed as in (\ref{S2.11}) to conclude part (vi) of Theorem \ref{thm1.2}. From this we also note that $\ttK(z)$ minimizes the integral (\ref{pt45}) if and only if
\begin{equation}
\ttK\left(n - \hh\right) = f_{\mu}\left(n - \hh\right)
\end{equation}
for all $n \in \Z$. Therefore
\begin{equation*}
(\ttK - K_{\mu})\left( n - \hh \right) = 0
\end{equation*}
for all $n \in \Z$. From the interpolation formulas (see \cite[vol II, p. 275]{Z} or \cite[p. 187]{V}) we observe that 
$$ (\ttK - K_{\mu})(z) =  \beta \cos(\pi z)$$
for some constant $\beta$. But we have seen that $(\ttK - K_{\mu})(x)$ is integrable, thus $\beta = 0$. This concludes the proof of (vii) in Theorem \ref{thm1.2}.


\section{Extremal trigonometric polynomials}

We consider in this section the problem of approximating certain real valued periodic functions by trigonometric polynomials of bounded degree.  We identify functions defined on $\R$ and having period $1$
with functions defined on the compact quotient group $\R/\Z$.  For real numbers $x$ we write
\begin{equation*}
\|x\| = \min\{|x - m|: m\in\Z\}
\end{equation*}
for the distance from $x$ to the nearest integer.  Then $\|\ \|:\R/\Z\rightarrow [0,\h]$ is well defined, and
$(x,y)\rightarrow \|x - y\|$ defines a metric on $\R/\Z$ which induces its quotient topology.  Integrals over 
$\R/\Z$ are with respect to Haar measure normalized so that $\R/\Z$ has measure $1$.  

Let $F:\C\rightarrow \C$ be an entire function of exponential type at most $\pi\delta$, where $\delta$ is a 
positive parameter, and assume that $x\mapsto F(x)$ is integrable on $\R$.  Then the Fourier transform
\begin{equation}\label{ef0}
\tF(t) = \int_{-\infty}^{\infty} F(x)e(-tx)\, \dx
\end{equation}
is a continuous function on $\R$.  By classical results of Plancherel and Polya \cite{PP} (see also 
\cite[Chapter 2, Part 2, section 3]{Y}) we have
\begin{equation}\label{ef1}
\sum_{m=-\infty}^{\infty} |F(\alpha_m)| \le C_1(\epsilon, \delta) \int_{-\infty}^{\infty} |F(x)|\, \dx,
\end{equation}
where $m\mapsto \alpha_m$ is a sequence of real numbers such that $\alpha_{m+1} - \alpha_m \ge \epsilon > 0$, and 
\begin{equation}\label{ef2}
\int_{-\infty}^{\infty} |F^{\prime}(x)|\, \dx \le C_2(\delta) \int_{-\infty}^{\infty} |F(x)|\, \dx.
\end{equation}
Plainly (\ref{ef1}) implies that $F$ is uniformly bounded on $\R$, and therefore $x\mapsto |F(x)|^2$ is integrable.
Then it follows from the Paley-Wiener theorem (see \cite[Theorem 19.3]{Rudin}) that $\tF(t)$ is supported
on the interval $\big[-\tfrac{\delta}{2}, \tfrac{\delta}{2}\big]$.  

The bound (\ref{ef2}) implies that $x\mapsto F(x)$ has bounded variation on $\R$.  Therefore the Poisson 
summation formula (see \cite[Volume I, Chapter 2, section 13]{Z}) holds as a pointwise identity
\begin{equation}\label{ef3}
\sum_{m=-\infty}^{\infty} F(x + m) = \sum_{n=-\infty}^{\infty} \tF(n) e(nx),
\end{equation}
for all real $x$.  It follows from (\ref{ef1}) that the sum on the left of (\ref{ef3}) is absolutely
convergent.  As the continuous function $\tF(t)$ is supported on $\big[-\tfrac{\delta}{2}, \tfrac{\delta}{2}\big]$, the sum
on the right of (\ref{ef3}) has only finitely many nonzero terms, and so defines a trigonometric polynomial
in $x$.

Next we consider the entire function $z\mapsto K\bigl(\delta^{-1}\lambda, \delta z\bigr)$.  This function has exponential type at most $\pi\delta$. We apply (\ref{ef3}) to obtain the identity
\begin{equation}\label{ef4}
\sum_{m=-\infty}^{\infty} K\bigl(\delta^{-1}\lambda, \delta(x + m)\bigr) 
	= \delta^{-1} \sum_{|n| \le \tfrac{\delta}{2}} \tK\bigl(\delta^{-1}\lambda, \delta^{-1}n\bigr) e(nx)
\end{equation}
for all real $x$, and for all positive values of the parameters $\delta$ and $\lambda$.  For our purposes it will 
be convenient to use (\ref{ef4}) with $\delta = 2N+2$, where $N$ is a nonnegative integer, and to modify
the constant term.  For each nonnegative integer $N$ we define a trigonometric polynomial $k(\lambda, N; x)$, of degree at most $N$, by
\begin{align}\label{ef6}
\begin{split}
k(\lambda, N; x) &= -\tfrac{2}{\lambda} + \tfrac{1}{2N+2} \sum_{n=-N}^N \tK\bigl(\tfrac{\lambda}{2N+2}, \tfrac{n}{2N+2}\bigr) e(nx).  
\end{split}
\end{align}

For $\lambda > 0$ the function $x\mapsto e^{-\lambda |x|}$ is continuous, integrable on $\R$, and has
bounded variation.  Therefore, the Poisson summation formula also provides the pointwise identity
\begin{equation}\label{ef10}
\sum_{m=-\infty}^{\infty} e^{-\lambda |x + m|} 
	= \sum_{n=-\infty}^{\infty} \frac{2\lambda}{\lambda^2 + 4\pi^2n^2}\,e(nx).
\end{equation}
And we find that
\begin{equation}\label{ef11}
\sum_{m=-\infty}^{\infty} e^{-\lambda |x + m|} 
	= \frac{\cosh\bigl(\lambda(x - [x] - \hh)\bigr)}{\sinh\bigl(\tfrac{\lambda}{2}\bigr)},
\end{equation}
where $[x]$ is the integer part of the real number $x$.  For our purposes it will be convenient to define
\begin{equation*}\label{ef12}
p:(0,\infty)\times \R/\Z \rightarrow \R 
\end{equation*}
by
\begin{equation}\label{ef13}
p(\lambda, x) = - \tfrac{2}{\lambda} + \sum_{m=-\infty}^{\infty} e^{-\lambda |x + m|}.
\end{equation}
Then $p(\lambda, x)$ is continuous on $(0,\infty)\times \R/\Z$, and differentiable with respect to $x$ at 
each non integer point $x$.  It follows from (\ref{ef10}) that the Fourier coefficients of 
$x\mapsto p(\lambda, x)$ are given by
\begin{equation}\label{ef15}
\int_{\R/\Z} p(\lambda, x)\, \dx = 0,
\end{equation}
and
\begin{equation}\label{ef16}
\int_{\R/\Z} p(\lambda, x)e(-nx)\, \dx = \frac{2\lambda}{\lambda^2 + 4\pi^2n^2}
\end{equation}
for integers $n\not= 0$. 

\begin{theorem}\label{thm6.1}  Let $\lambda$ be a positive real number and $N$ a nonnegative integer. 
\begin{itemize}
\item[(i)]  If $\wk(x)$ is a trigonometric polynomial of degree at most $N$ then
\begin{equation}\label{ef24}
\int_{\R/\Z} \left| p(\lambda,x) - \wk(x) \right| \dx \geq \tfrac{2}{\lambda} - \tfrac{2}{\lambda} \sech \bigl( \tfrac{\lambda}{4N+4} \bigr)
\end{equation}
with equality if and only if $\wk(x) = k(\lambda, N; x)$.
\item[(ii)] For $x \in \R/\Z$ we have
\begin{equation}\label{Sec6.1.1}
\sgn(\cos \pi (2N+2) x) = \sgn\left\{p(\lambda,x) - k(\lambda, N; x) \right\}.
\end{equation}
\end{itemize}
\end{theorem}
\begin{proof} Throughout this proof we consider $\delta = 2N+2$. From (\ref{ef4}), (\ref{ef6}), (\ref{ef10}) and (\ref{ef13}) we obtain
\begin{align}\label{Sec6.2.1.1}
\begin{split}
 p(\lambda,x) - k(\lambda, N; x) &= \sum_{n=-\infty}^{\infty} \left\{ \frac{2\lambda}{\lambda^2 + 4\pi^2n^2} - \delta^{-1} \tK(\delta^{-1}\lambda, \delta^{-1} n) \right\} e(nx) \\
& = \sum_{m=-\infty}^{\infty} \left\{ e^{-\lambda |x + m|} - K\bigl(\delta^{-1}\lambda, \delta(x + m)\bigr) \right\}
\end{split}
 \end{align}
for all $x \in \R/\Z$. Identity (\ref{Sec6.1.1}) now follows from (\ref{Sec6.2.1.1}) and (\ref{Intro3}). Using now (\ref{Intro3}) and (\ref{Intro2}) we arrive at
\begin{align}
\begin{split}
\int_{\R/\Z} \bigl| p(\lambda,x)  - & k(\lambda, N; x) \bigr| \dx = \\
& = \int_{\R/\Z} \left|\sum_{m=-\infty}^{\infty} \left\{ e^{-\lambda |x + m|} - K\bigl(\delta^{-1}\lambda, \delta(x + m)\bigr) \right\} \right| \dx\\
& = \int_{\R/\Z} \sum_{m=-\infty}^{\infty} \left| e^{-\lambda |x + m|} - K\bigl(\delta^{-1}\lambda, \delta(x + m)\bigr) \right| \dx\\
& = \int_{-\infty}^{\infty} \left| e^{-\lambda |x|} - K (\delta^{-1}\lambda, \delta x) \right| \dx \\
& = \frac{2}{\lambda} - \frac{2}{\lambda} \sech \Bigl( \frac{\lambda}{2\delta} \Bigr)\,,
\end{split}
\end{align}
and this proves that equality happens in (\ref{ef24}) when $\wk(x) = k(\lambda, N; x)$. 

Now let $\wk(x)$ be a general trigonometric polynomial of degree at most $N$. Using identity (\ref{S2.10}) we obtain
\begin{align}\label{Sec6.3.1}
\begin{split}
  \int_{\R/\Z} \bigl| p(\lambda,x) - & \wk(x) \bigr| \dx  \geq \left| \int_{\R/\Z} \bigl(p(\lambda,x) - \wk(x)\bigr) \sgn\{\cos \pi \delta x\}\, \dx \right|\\
& = \left| \int_{\R/\Z} p(\lambda,x) \sgn\{\cos \pi \delta x\}\, \dx \right|\\
& = \left| \frac{2}{\pi} \sum_{k = -\infty}^{\infty} \frac{(-1)^k}{(2k+1)} \int_{\R/\Z} p(\lambda,x) e \bigl((k + \hh) \delta x\bigr)\, \dx \right|\\
& = \left| \frac{2}{\pi} \sum_{k = -\infty}^{\infty} \frac{(-1)^k}{(2k+1)} \frac{2\lambda}{\left( \lambda^2 + 4 \pi^2 \bigl((k + \hh) \delta\bigr)^2 \right)} \right|\\
& = \frac{2}{\lambda} - \frac{2}{\lambda} \sech \Bigl( \frac{\lambda}{2\delta} \Bigr)
\end{split}
\end{align}
which proves (\ref{ef24}). If equality happens in (\ref{Sec6.3.1}) we must have (recall that $\delta = 2N+2$)
\begin{equation}\label{Sec6.3.2}
\wk\left(\tfrac{1}{2N+2} (k + \hh)\right) = p\left(\lambda,\tfrac{1}{2N+2} (k + \hh)\right) \ \ \ \textrm{for} \ \ \ k = 0,1,2,...,2N+1.
\end{equation}
Since the degree of $\wk(x)$ is at most $N$, such polynomial exists and is unique \cite[Vol II, page 1]{Z}. Observe that $k(\lambda, N; x)$ already satisfies (\ref{Sec6.3.2}), this being a consequence of (\ref{Sec6.1.1}). Therefore, we must have $\wk(x) = k(\lambda, N; x)$, which finishes the proof.

\end{proof}
It follows from (\ref{ef11}) and (\ref{ef13}) that
\begin{equation}\label{ef42}
-\big\{\tfrac{2}{\lambda} - \csch\bigl(\tfrac{\lambda}{2}\bigr)\big\} = p(\lambda, \hh) \le p(\lambda, x)
	\le p(\lambda, 0) = \coth\bigl(\tfrac{\lambda}{2}\bigr) - \tfrac{2}{\lambda}.
\end{equation}
Then (\ref{ef42}) provides the useful inequality
\begin{align}\label{ef43}
\begin{split}
\bigl|p(\lambda, x)\bigr| &\le \bigl|p(\lambda, x) - p(\lambda, \hh)\bigr| + \bigl|p(\lambda, \hh)\bigr| \\
	&= p(\lambda, x) - p(\lambda, \hh) - p(\lambda, \hh) \\
	&= p(\lambda, x) + 2\big\{\tfrac{2}{\lambda} - \csch\bigl(\tfrac{\lambda}{2}\bigr)\big\}
\end{split}
\end{align}
at each point $(\lambda, x)$ in $(0, \infty)\times \R/\Z$.  From (\ref{ef15}) and (\ref{ef43}) we conclude that
\begin{equation}\label{ef44}
\int_{\R/\Z} \bigl|p(\lambda, x)\bigr|\, \dx \le 2\big\{\tfrac{2}{\lambda} - \csch\bigl(\tfrac{\lambda}{2}\bigr)\big\}. 
\end{equation}

Let $\mu$ be a measure on the Borel subsets of $(0,\infty)$ that satisfies (\ref{Intro4}).  For $0 < x < 1$ it 
follows from (\ref{ef11}) and (\ref{ef13}) that $\lambda\mapsto p(\lambda, x)$ is integrable on $(0,\infty)$ 
with respect to $\mu$. We define $q_{\mu}:\R/\Z\rightarrow \R\cup\{\infty\}$ by
\begin{equation}\label{ef50}
q_{\mu}(x) = \int_0^{\infty} p(\lambda, x)\, \dmu,
\end{equation}
where
\begin{equation}\label{ef51}
q_{\mu}(0) = \int_0^{\infty}\big\{\coth\bigl(\tfrac{\lambda}{2}\bigr) - \tfrac{2}{\lambda}\big\}\, \dmu
\end{equation}
may take the value $\infty$.  Using (\ref{ef44}) and Fubini's theorem we have
\begin{align*}\label{ef52}
\begin{split}
\int_{\R/\Z} \bigl|q_{\mu}(x)\bigr|\ \dx &\le \int_0^{\infty} \int_{\R/\Z} \bigl|p(\lambda, x)\bigr|\, \dx\, \dmu \\ 
	&\le 2 \int_0^{\infty} \big\{\tfrac{2}{\lambda} - \csch\bigl(\tfrac{\lambda}{2}\bigr)\big\}\, \dmu < \infty,
\end{split}
\end{align*}
so that $q_{\mu}$ is integrable on $\R/\Z$.  Using (\ref{ef15}) and (\ref{ef16}), we find that the Fourier 
coefficients of $q_{\mu}$ are given by
\begin{equation}\label{ef53}
\tq_{\mu}(0) = \int_{\R/\Z} q_{\mu}(x)\ \dx = \int_0^{\infty} \int_{\R/\Z} p(\lambda, x)\, \dx\, \dmu = 0,
\end{equation}
and
\begin{align}\label{ef54}
\begin{split}
\tq_{\mu}(n) &= \int_{R/\Z} q_{\mu}(x) e(-nx)\, \dx \\
	     &= \int_0^{\infty} \int_{\R/\Z} p(\lambda, x)e(-nx)\, \dx\, \dmu \\
 	     &= \int_0^{\infty} \frac{2\lambda}{\lambda^2 + 4\pi^2n^2}\, \dmu,
\end{split}
\end{align}
for integers $n\not= 0$.  As $n\mapsto \tq_{\mu}(n)$ is an even function of $n$, and $\tq_{\mu}(n) \ge \tq_{\mu}(n+1)$ 
for $n \geq 1$, the partial sums
\begin{equation}\label{ef55}
q_{\mu}(x) = \lim_{N\rightarrow \infty} \sum_{\substack{n=-N\\ n\not= 0}}^N \tq_{\mu}(n)e(nx)
\end{equation}
converge uniformly on compact subsets of $\R/\Z\setminus\{0\}$, (see \cite[Chapter I, Theorem 2.6]{Z}).
In particular, $q_{\mu}(x)$ is continuous on $\R/\Z\setminus\{0\}$. 

For each nonnegative integer $N$, we define a trigonometric polynomial 
$k_{\mu}(N; x)$, of degree at most $N$, by
\begin{equation}\label{ef60}
k_{\mu}(N; x) = \sum_{n = -N}^N \uk_{\mu}(N; n)e(nx),
\end{equation}
where the Fourier coefficients are given by (recall here Lemma \ref{lem3.3})
\begin{align}\label{ef61}
\begin{split}
\uk_{\mu}(N; 0) & = \int_0^{\infty} \bigl\{ - \tfrac{2}{\lambda} + \tfrac{1}{2N+2}\, \tK\bigl(\tfrac{\lambda}{2N+2},0\bigr)\bigr\}\, \dmu\\
& = - \int_0^{\infty} \bigl\{ \tfrac{2}{\lambda} - \tfrac{1}{2N+2}\, \csch \big(\tfrac{\lambda}{4N+4}\bigr) \bigr\}\, \dmu
\end{split}
\end{align}
and
\begin{equation}\label{ef62}
\uk_{\mu}(N; n) = \tfrac{1}{2N+2} \int_0^{\infty} \tK\bigl(\tfrac{\lambda}{2N+2}, \tfrac{n}{2N+2}\bigr)\ \dmu,
\end{equation}
for $n\not= 0$.  

\begin{theorem}\label{thm6.2} Let $N$ be a nonnegative integer and assume that $\mu$ satisfies {\rm (\ref{Intro4})}.
\begin{itemize}
\item[(i)]  If $\wk(x)$ is a trigonometric polynomial of degree at most $N$, then
\begin{equation}\label{ef63}
\int_{\R/\Z} \left| q_{\mu}(x) - \wk(x) \right| \dx \geq \int_0^{\infty} \bigl\{ \tfrac{2}{\lambda} - \tfrac{2}{\lambda} \sech \bigl( \tfrac{\lambda}{4N+4} \bigr) \bigr\} \, \dmu
\end{equation}
with equality if and only if $\wk(x) = k_{\mu}(N; x)$.
\item[(ii)] For $x \in \R/\Z$ we have
\begin{equation}\label{ef64}
\sgn(\cos \pi (2N+2) x) = \sgn\left\{q_{\mu}(x) - k_{\mu}(N; x) \right\}.
\end{equation}
\end{itemize}
\end{theorem}

\begin{proof}
We use the elementary identity 
\begin{equation}\label{ef65}
k_{\mu}(N; x) = \int_0^{\infty} k(\lambda, N; x)\, \dmu.
\end{equation}
Expression (\ref{Sec6.1.1}), together with (\ref{ef50}) and (\ref{ef65}), imply (\ref{ef64}). Using (\ref{Sec6.1.1}) and (\ref{ef24}) we observe that
\begin{align}
 \begin{split}
\int_{\R/\Z} \left| q_{\mu}(x) -  k_{\mu}(N; x) \right| \dx & = \int_{\R/\Z} \left| \int_0^{\infty} \bigl\{p(\lambda, x) - k(\lambda, N; x)\bigr\}\, \dmu \right| \dx \\
& = \int_{\R/\Z}  \int_0^{\infty} \left| p(\lambda, x) - k(\lambda, N; x) \right|\, \dmu\,  \dx \\
& = \int_0^{\infty} \int_{\R/\Z} \left| p(\lambda, x) - k(\lambda, N; x) \right|\,\dx \, \dmu\\
& = \int_0^{\infty} \bigl\{ \tfrac{2}{\lambda} - \tfrac{2}{\lambda} \sech \bigl( \tfrac{\lambda}{4N+4} \bigr) \bigr\} \, \dmu.
 \end{split}
\end{align}
This proves that equality happens in (\ref{ef63}) when $\wk(x) = k_{\mu}(N; x)$. The proof of the lower bound (\ref{ef63}) and the uniqueness part are similar to the ones given in Theorem \ref{thm6.1}.
\end{proof}

\end{document}